\documentclass[reqno]{amsart}
\usepackage{amsmath, amsthm, amssymb, mathabx}
\begin{document}
\newtheorem{theorem}{Theorem}
\newtheorem{proposition}{Proposition}
\newtheorem{lemma}{Lemma}
\newtheorem{corollary}{Corollary}

\theoremstyle{definition}

\newtheorem{definition}{Definition}
\newtheorem{example}{Example}
\newtheorem{exercise}{Exercise}
\newtheorem{remark}{Remark}

%
%
\thanks{}
\title[Ground state solution...]{Ground state solution of a noncooperative elliptic system}

\author[C. J. Batkam]{Cyril Joel Batkam}
\address{Cyril Joel Batkam 
 \newline
D\'epartement de math\'ematiques,
\newline
Universit\'e de Sherbrooke,
\newline
Sherbrooke, Qu\'ebec, J1K 2R1, CANADA.}
\numberwithin{equation}{section}

\maketitle
\begin{abstract}
        In this paper, we study the existence of a ground state solution, that is, a non trivial solution with least energy, of a noncooperative semilinear elliptic system on a bounded domain. By using the method of the generalized Nehari manifold developed recently by Szulkin and Weth, we prove the existence of a ground state solution when the nonlinearity is subcritical and satisfies a weak superquadratic condition.
\end{abstract}

%
\section{Introduction}
In this paper, we are concerned with the following noncooperative elliptic system
\begin{equation*}
    (\mathcal{P}) \,\,\ \left\{
                          \begin{array}{ll}
                            -\Delta u=F_u(x,u,v) , \,\ x\in \Omega & \hbox{} \\ \\
                          \,\,\  \Delta v=F_v(x,u,v) , \,\ x\in\Omega & \hbox{} \\ \\
                            u=v=0 , \,\ x\in \partial\Omega,& \hbox{}
                          \end{array}
                        \right.
\end{equation*}
where $\Omega$ is a bounded smooth domain in $\mathbb{R}^N$ and $F_u$ designates the partial derivative  with respect to $u$ of the nonlinearity $F:\overline{\Omega}\times\mathbb{R}^2\rightarrow\mathbb{R}$. The solutions of such systems are steady state of reaction-diffusion systems which arise in many applications such as Chemistry, Biology, Geology, Physics or Ecology.
It is well known $(\mathcal{P})$ has variational structure, that is, its solutions can be found as critical points of the following functional
\begin{equation*}
    \Phi(u,v):=\int_\Omega \Bigl(\frac{1}{2}|\nabla u|^2- \frac{1}{2}|\nabla v|^2-F(x,u,v)\Bigr)
\end{equation*}
defined on $H_0^1(\Omega)\times H_0^1(\Omega)$ \big(i.e the solutions of the equation $\Phi'(u,v)=0$, where $\Phi'$ is the Fr\'{e}chet derivative of $\Phi$\big). In this paper, we will be interested in the existence of a ground state solution, that is, a non trivial solution which minimizes the energy functional $\Phi$. Let us recall that ground state solutions play an important role in applications. For instance, in the study of the formation of spacial patterns in various reaction-diffusion systems, the solutions of the system often converge to a ground state of a simplified semilinear elliptic system, as time tends to infinity (see \cite{K}).
\par In recent years, the existence of ground state solutions of elliptic equations and systems has been widely study, and many interesting results have been obtained (see for instance \cite{K, Mo, S-W, Yang, Chen-Ma, Sch} and the references therein). In (\cite{S-W}, chapter $3$), the authors presented the well known method of the Nehari manifold in a unified way, which can be applied to find ground state solutions of the following elliptic system of cooperative type:
\begin{equation*}
   \left\{ \begin{array}{ll}
                            -\Delta u=F_u(x,u,v), \,\ x\in\Omega & \hbox{} \\ \\
                            -\Delta v=F_v(x,u,v), \,\ x\in \Omega & \hbox{} \\ \\
                            u=v=0 , \,\ x\in \partial\Omega.& \hbox{}
                          \end{array}
                        \right.
\end{equation*}
However, there appears to be no result in the noncooperative case.
\par Let us now introduce the precise assumptions on the nonlinearity $F$ under which our problem is studied:
\begin{itemize}
  \item [$(F_1)$] $F\in\mathcal{C}^1(\overline{\Omega}\times\mathbb{R}^2,\mathbb{R})$ and $F(x,0)=0$ for every $x$ in $\overline{\Omega}$, and $0\in\mathbb{R}^2$.\\
  \item [$(F_2)$] $|\nabla F(x,u)|\leq a\big(1+|u|^{p-1}\big)$, for some $p\in(2,2^\star)$, $x\in\Omega$, $u=(u_1,u_2)\in\mathbb{R}^2$, where $2^\star:=2N/(N-2)$ if $N\geq3$ and $2^\star:=\infty$ if $N=1,2.$ \\
  \item [$(F_3)$]  $F(x,u)=\circ(|u|^2)$ as $|u|\rightarrow0$, uniformly in $x$. \\
  \item [$(F_4)$] $\frac{F(x,u)}{|u|^2}\rightarrow\infty$ as $|u|\rightarrow\infty$, uniformly in $x$.\\
  \item [$(F_5)$] $F(x,u)>0$ and $u\cdot\nabla F(x,u)>2F(x,u)$, $\forall u\in\mathbb{R}^2\backslash\{0\}$.\\
  \item [$(F_6)$] $\big(v\cdot\nabla F(x,u)\big)(u\cdot v)\geq0$, $\forall v\in\mathbb{R}^2$.\\
  \item [$(F_7)$] If $|u|=|v|$, then $F(x,u)=F(x,v)$ and $v\cdot\nabla F(x,u)\leq u\cdot\nabla F(x,u)$, with strict inequality if in addition $u\neq v$.\\
  \item [$(F_8)$] $|u|\neq |v|$ and $u\cdot v\neq0$ $\Rightarrow$ $v\cdot \nabla F(x,u)\neq u\cdot\nabla F(x,v)$.
\end{itemize}
Where we write $F(x,u)=\circ(|u|^2)$ as $|u|\rightarrow0$ to mean that $\lim\limits_{|u|\rightarrow0}\frac{F(x,u)}{|u|^2}=0$. $\nabla F(x,u)$ denotes the gradient of $F$ with respect to $u$ and $u\cdot v$ is the usual inner product in $\mathbb{R}^2$. A simple example of a nonlinearity satisfying these conditions is $F(x,u)=f(x)|u|^p$, where $2<p<2^\star $ and $f>0$ is of class $\mathcal{C}^1$ on $\overline{\Omega}$. \\
\par The main result of this paper is the following:
\begin{theorem}\label{maintheorem}
Under assumptions $(F_1)-(F_8)$, $(\mathcal{P})$ has a ground state solution.
\end{theorem}
\par We point out here that the energy functional $\Phi$ associated to $(\mathcal{P})$ is strongly indefinite, in the sense that the negative and positive eigenspaces of its quadratic part are both infinite-dimensional. Therefore, the set
$$\mathcal{N}:=\big\{z\in H_0^1(\Omega)\times H_0^1(\Omega)\, \big|\, z\neq0 \, \textnormal{and }\big<\Phi'(z),z\big>=0\big\}$$
need not be closed (since $\inf_\mathcal{N}\Phi$ can be $0$), and Theorem \ref{maintheorem} cannot be proved by using the usual method of the Nehari manifold (see \cite{S-W}, chapter $3$ for a description and some applications of this method). To circumvent the difficulty posed by the strongly indefiniteness of $\Phi$, we will use the method of the generalized Nehari manifold inspired by Pankov \cite{Pan}, and developed recently by Szulkin and Weth \cite{S-W}, which consists in a reduction into two steps.
\par We organize the paper in the following way: In section \ref{section1}, the method of the generalized Nehari manifold is briefly presented while in section \ref{section2}, the existence of a ground state solution is proved.
%

\section{The method of the generalized Nehari manifold}\label{section1}

Let $X$ be a Hilbert space with norm $\|\cdot\|$, and an orthogonal decomposition $X=X^+\oplus X^-$. We denote by $S^+$ the unit sphere in $X^+$; that is,
\begin{equation*}
    S^+:=\big\{u\in X^+\,\big|\,\|u\|=1\big\}.
\end{equation*}
For $u=u^+ + u^-\in X$, where $u^\pm\in X^\pm$, we define
\begin{equation}\label{}
    X(u):=\mathbb{R}u\oplus X^-\equiv\mathbb{R}u^+\oplus X^- \,\,\, \textnormal{and}\,\,\, \widehat{X}(u):=\mathbb{R}^+u\oplus X^-\equiv\mathbb{R}^+u^+\oplus X^-,
\end{equation}
where $\mathbb{R}v:=\{\lambda v\,;\,\lambda\in\mathbb{R}\}$ and $\mathbb{R}^+v:=\{\lambda v \,;\,\lambda\geq0\}$ for $v\in X$.
\par Let $\Phi$ be a $\mathcal{C}^1-$functional defined on $X$ by
\begin{equation*}
    \Phi(u):=\frac{1}{2}\|u^+\|^2-\frac{1}{2}\|u^-\|^2-I(u).
\end{equation*}

We consider the following situation:

\begin{enumerate}
  \item [$(A_1)$]  $I(0)=0$, $\frac{1}{2}\big<I'(u),u\big>>I(u)>0$ for all $u\neq0$, and $I$ is weakly lower semicontinuous.
  \item [$(A_2)$] For each $w\in X\backslash X^-$ there exists a unique nontrivial critical point $\widehat{m}(w)$ of $\Phi|_{\widehat{X}(w)}$. Moreover, $\widehat{m}(w)$ is the unique global maximum of $\Phi|_{\widehat{X}(w)}$.
  \item [$(A_3)$] There exists $\delta>0$ such that $\|\widehat{m}(w)^+\|\geq\delta$ for all $w\in X\backslash X^-$, and for each compact subset $\mathcal{K}\subset X\backslash X^-$ there exists a constant $C_\mathcal{K}$ such that $\|\widehat{m}(w)\|\leq C_\mathcal{K}$.
\end{enumerate}
\par We consider the following set introduced by Pankov \cite{Pan}:
\begin{equation*}
    \mathcal{M}:=\big\{u\in X\backslash X^-: \big<\Phi'(u),u\big>=0 \,\, \textnormal{and} \,\, \big<\Phi'(u),v\big>=0\,\, \forall v\in X^-\big\}.
\end{equation*}
Following Szulkin and Weth \cite{S-W}, we will call $\mathcal{M}$ the generalized Nehari manifold.
\begin{remark}
 \textnormal{By $(A_1)$ $\mathcal{M}$ contains all nontrivial critical points of $\Phi$ and by $(A_2)$ $\widehat{X}(w)\cap\mathcal{M}=\{\widehat{m}(w)\}$ whenever $w\in X\backslash X^-$.}
 \end{remark}

\par In the following we consider the mappings:
\begin{equation*}
    \widehat{m}:X\backslash X^-\rightarrow \mathcal{M}, \, w\mapsto \widehat{m}(w)\,\,\,\, \textnormal{and }\,\,\, m:=\widehat{m}|_{S^+}.
\end{equation*}
The following results are due to A. Szulkin and T. Weth \big(\cite{S-W}, Chapter $4$\big). For reader's convenience we provide the proofs here.
\begin{proposition}\label{prop1}
If $(A_1),$ $(A_2)$ and $(A_3)$ are satisfied, then
\begin{itemize}
  \item [(a)] $\widehat{m}$ is continuous,
  \item [(b)] $m$ is a homeomorphism between $S^+$ and $\mathcal{M}$.
\end{itemize}
\end{proposition}
\begin{proof}
\item[(a)] Let $(w_n)\subset X\backslash X^-$ such that $w_n\rightarrow w\notin X^- $. We want to show that $\widehat{m}(w_n)\rightarrow \widehat{m}(w)$. Since $\widehat{m}(w_n)=\widehat{m}(w_n^+/\|w_n^+\|)$, we may assume without loss of generality that $w_n\in S^+$. Therefore, it suffices to show that $\widehat{m}(w_n)\rightarrow \widehat{m}(w)$ after passing to a subsequence. Write $\widehat{m}(w_n)=s_nw_n+v_n$, with $s_n\geq0$ and $v_n\in X^-$. By $(A_3)$, the sequence $(\widehat{m}(w_n))$ is bounded. So taking a subsequence, we have $s_n\rightarrow s$ and $v_n\rightharpoonup v$. Setting $\widehat{m}(w)=sw+v$, it follows from $(A_2)$ that
    \begin{equation*}
        \Phi(\widehat{m}(w_n))\geq \Phi(s_nw_n+v)\rightarrow\Phi(sw+v)=\Phi(\widehat{m}(w))
    \end{equation*}
    and hence, using the weak lower semicontinuity of the norm and $I$,
\begin{align*}
    \Phi(\widehat{m}(w))&\leq\lim\limits_{n\rightarrow\infty}\Phi(\widehat{m}(w_n)) = \lim\limits_{n\rightarrow\infty}\Big(\frac{1}{2}s_n^2-\frac{1}{2}\|v_n\|^2-I(\widehat{m}(w_n))\Big) \\
       &\leq \frac{1}{2}s^2-\frac{1}{2}\|v\|^2-I(sw+v)\leq \Phi(\widehat{m}(w)).
\end{align*}
  Hence the inequalities above must be equalities. It follows that $(v_n)$ is strongly convergent, and so $v_n\rightarrow v$. Hence $\widehat{m}(w_n)=s_nw_n+v_n\rightarrow sw+v=\widehat{m}(w)$.
  \item[(b)] It is easy to see that $m$ is a bijection whose inverse $m^{-1}$ is given by
  \begin{equation*}
    m^{-1}(u)=\frac{u^+}{\|u^+\|},\quad \forall u\in\mathcal{M}.
  \end{equation*}
  Since $m^{-1}$ is clearly continuous, we then deduce from (a) that $m$ is a homeomorphism between $S^+$ and $\mathcal{M}$. \medskip
\end{proof}
\par Let
\begin{equation*}
    \widehat{\Psi}:X^+\backslash\{0\}\rightarrow \mathbb{R},\, \widehat{\Psi}(w):=\Phi(\widehat{m}(w)) \,\, \textnormal{ and }\,\, \Psi:=\widehat{\Psi}|_{S^+}.
\end{equation*}

\begin{proposition}\label{prop2}
Under assumptions $(A_1),$ $(A_2)$ and $(A_3)$, $\widehat{\Psi}$ is of class $\mathcal{C}^1$ and
\begin{equation*}
    \big<\widehat{\Psi}'(w),z\big>=\frac{\|\widehat{m}(w)^+\|}{\|w\|}\big<\widehat{\Phi}'(w),z\big>,\,\, \textnormal{for all}\,\,w,z\in X^+,\,\, w\neq0.
\end{equation*}
\end{proposition}
\begin{proof}
Let $w\in X^+\backslash\{0\}$, $z\in X^+$ and put $\widehat{m}(w)=s_ww+v_w$, $v_w\in X^-$. Using the maximality property of $\widehat{m}(w)$ given by $(A_2)$ and the mean value theorem, we obtain
\begin{align*}
  \widehat{\Psi}(w+tz)-\widehat{\Psi}(w) &= \Phi(s_{w+tz}(w+tz)+v_{w+tz})-\Phi(s_ww+v_w) \\
   &\leq \Phi(s_{w+tz})(w+tz)+v_{w+tz})-\Phi(s_{w+tz})w+v_{w+tz}) \\
   &= \big<\Phi'(s_{w+tz}w+v_{w+tz}+\tau_ts_{w+tz}tz),s_{w+tz}tz\big>,
\end{align*}
where $|t|$ is small enough and $\tau_t\in(0,1)$. Similarly,
\begin{align*}
  \widehat{\Psi}(w+tz)-\widehat{\Psi}(w) &\geq  \Phi(s_{w}(w+tz)+v_{w})-\Phi(s_ww+v_w)\\
   &=  \big<\Phi'(s_ww+v_w+\eta_ts_wtz),s_wtz\big>,
\end{align*}
where $\eta_t\in(0,1)$. Since the mappings $w\mapsto s_w$ and $w\mapsto v_w$ are continuous according to Proposition \ref{prop1}, we see by combining these two inequalities that
\begin{align*}
  \big<\widehat{\Psi}'(w),z\big> &= \lim\limits_{t\rightarrow0}\frac{\widehat{\Psi}(w+tz)-\widehat{\Psi}(w)}{t}\\
   &=  s_w\big<\Phi'(s_ww+v_w),z\big>=\frac{\|\widehat{m}(w)^+\|}{\|w\|}\big<\Phi'(\widehat{m}(w)),z\big>.
\end{align*}
Hence the G\^{a}teaux derivative of $\widehat{\Psi}$ is bounded linear in $z$ and continuous in $w$. It follows from Proposition $1.3$ in \cite{W} that $\widehat{\Psi}$ is of class $\mathcal{C}^1$.\medskip
\end{proof}
\par Before giving a consequence of the previous propositions, which is the main result of this section, we recall some definitions.
\begin{definition}
Let $\varphi\in\mathcal{C}^1(X,\mathbb{R})$.
\begin{enumerate}
  \item A sequence $(u_n)\subset X$ is a Palais-Smale sequence \big(resp. a Palais-Smale sequence at level $c\in\mathbb{R}$\big) for $\varphi$ if $(\varphi(u_n))$ is bounded \big(resp. $\varphi(u_n)\rightarrow c$\big) and $\varphi'(u_n)\rightarrow0$ as $n\rightarrow\infty$.
  \item We say that $\varphi$ satisfies the Palais-Smale condition \big(resp. the Palais-Smale condition at level $c$\big) if every Palais-Smale sequence \big(resp. every Palais-Smale sequence at level $c$\big) has a convergent subsequence.
\end{enumerate}
\end{definition}
\begin{corollary}\label{reduction}
Assume that $(A_1),$ $(A_2)$ and $(A_3)$ are satisfied. Then:
\begin{itemize}
  \item [(a)] $\Psi\in\mathcal{C}^1(S^+,\mathbb{R})$ and
  \begin{equation*}
    \big<\Psi'(w),z\big>=\|m(w)^+\|\big<\Phi'(m(w)),z\big> \textnormal{ for all }z\in T_w(S),
  \end{equation*}
  where $T_w(S)$ is the tangent space of $S$ at $w$.
  \item [(b)] If $(w_n)$ is a Palais-Smale sequence for $\Psi,$ then $(m(w_n))$ is a Palais-Smale sequence for $\Phi$. If $(u_n)\subset\mathcal{M}$ is a bounded Palais-Smale sequence for $\Phi$, then $(m^{-1}(w_n))$ is a Palais-Smale sequence for $\Psi$.
  \item [(c)] $w$ is a critical point of $\Psi$ if and only if $m(w)$ is a nontrivial critical point of $\Phi$. Moreover, the corresponding critical values coincide and $inf_{S^+}\Psi=inf_\mathcal{M}\Phi$.
\end{itemize}
\end{corollary}
\begin{proof}
\item[(a)] This is a direct consequence of Proposition \ref{prop2}, since $m(w)=\widehat{m}(w)$ for $w\in S^+$.
\item[(b)] Let $(w)\subset S^+$ and let $u=m(w)\in\mathcal{M}$. We have an orthogonal decomposition
\begin{equation*}
    X=X(w)\oplus T_{w}(S^+)=X(u)\oplus T_{w}(S^+).
\end{equation*}
Using (a) we have
\begin{equation}\label{etoile}
     \|\Psi'(w)\|= \sup_{\substack{z\in T_{w}(S^+)\\ \|z\|=1}}\big<\Psi'(w),z\big>=\sup_{\substack{z\in T_{w}(S^+)\\ \|z\|=1}}\|u^+\|\big<\Phi'(m(w)),z\big> = \|u^+\|\|\Phi'(u)\|,
\end{equation}
where the last equality holds because $\big<\Phi'(u),v\big>=0$ for all $v\in X(w)$, and $T_{w}(S^+)$ is orthogonal to $X(u)$. By $(A_3)$, there is $\delta>0$ such that $\|u^+\|\geq\delta$. It is then easy to conclude.
\item[(c)] By \eqref{etoile}, $\Psi'(w)=0$ if and only if $\Phi'(m(w))=0$. The other part is clear. \medskip
\end{proof}

\section{Proof of the main result}\label{section2}
Let $X:=H_0^1(\Omega)\times H_0^1(\Omega)$ endowed with the norm
\begin{equation*}
    \|(a,b)\|=\big(\|\nabla a\|_{L^2(\Omega)}^2+\|\nabla b\|_{L^2(\Omega)}^2\big)^\frac{1}{2},
\end{equation*}
 which by the Poincar\'{e} inequality is equivalent to its usual norm. Define
\begin{equation*}
    X^+:=H_0^1(\Omega)\times\{0\} \,\,\, \textnormal{and} \,\,\, X^-:=\{0\}\times H_0^1(\Omega).
\end{equation*}
Then for $u=u^++u^-\in X$, we have
\begin{equation}\label{phi}
    \Phi(u)=\frac{1}{2}\|u^+\|^2-\frac{1}{2}\|u^-\|^2-I(u),
\end{equation}
where $I(u):=\int_\Omega F(x,u)dx.$ \\
We recall that for $u\in X$,
\begin{equation}\label{d}
    X(u):=\mathbb{R}u\oplus X^-\equiv \mathbb{R}u^+\oplus X^- \,\,\, \textnormal{and} \,\,\, \widehat{X}(u):=\mathbb{R}^+u\oplus X^-\equiv \mathbb{R}^+u^+\oplus X^-.
\end{equation}
\par By a standard argument we have:
\begin{lemma}
 Under $(F_1)-(F_2)$, $\Phi\in\mathcal{C}^1(X,\mathbb{R})$ and
\begin{equation}\label{phiprime}
    \big<\Phi'(u),v\big>=\int_\Omega\Big(\nabla u^+\cdot\nabla v^+-\nabla u^-\cdot\nabla v^-- v\cdot\nabla F(x,u)\Big).
\end{equation}
\end{lemma}
Before giving the proof of the main theorem, we need some preliminary results.
\begin{lemma}\label{a1}
Assume $(F_1)$ and $(F_5)$. Then $(A_1)$ is satisfied.
\end{lemma}
\begin{proof}
Clearly by $(F_1)$ and $(F_5)$ we have $I(0)=0$ and $\frac{1}{2}\big<I'(u),u\big>>I(u)>0$, $\forall u\neq0.$ Let $(u_n)\subset X$ and $c\in\mathbb{R}$ such that $u_n\rightharpoonup u$ and $I(u_n)\leq c$. By Rellich-Kondrachov theorem $u_n\rightarrow u$ in $L^2(\Omega)\times L^2(\Omega)$, and taking a subsequence if necessary we have $u_n(x)\rightarrow u(x)$ a.e on $\Omega$. Since $F$ is continuous, we conclude by applying Fatou's Lemma that $I$ is weakly lower semicontinuous. \medskip
\end{proof}

\begin{lemma}\label{a2}
Under $(F_1)$, $(F_3)-(F_8)$, $(A_2)$ is satisfied.
\end{lemma}
\begin{proof}
\item[\textbf{(1)}] We first show that $\widehat{X}(w)\cap\mathcal{M}\neq\emptyset$ for any $w\in X\backslash X^-$.\\
              Let $w\in X\backslash X^-.$ Then $\Phi\leq0$ on $\widehat{X}(w)\backslash B_R$ for $R$ large enough, where $B_R:=\{u\in X\, |\, \|u\|\leq R\}$. In fact, if this is not true then there exists a sequence $(u_n)\subset\widehat{X}(w)$ such that $\|u_n\|\rightarrow\infty$ and $\Phi(u_n)>0$. Up to a subsequence we have $v_n=u_n/\|u_n\|\rightharpoonup v$.
              By \eqref{phi} we have
              \begin{equation*}
                0<\frac{\Phi(u_n)}{\|u_n\|^2}=\frac{1}{2}\|v_n^+\|^2-\frac{1}{2}\|v_n^-\|^2-\int_\Omega\frac{F(x,\|u_n\|v_n)}{\big|v_n\|u_n\|\big|^2}|v_n|^2.
              \end{equation*}
If $v\neq0$ we deduce, by using Fatou's Lemma and $(F_4)$, that $0\leq-\infty$; a contradiction. Consequently $v=0$. Since $\widehat{X}(w)=\widehat{X}(w^+/\|w^+\|)$, we may assume that $w\in S^+$. Now since $I(u_n)\geq0$ and $1=\|v_n^+\|^2+\|v_n^-\|^2$, then necessarily $v_n^+=s_nw\nrightarrow0$. Hence there is $r>0$ such that $\|v_n^+\|=\|s_nw\|>r$ $\forall n$. So $\|v_n^+\|=s_n$ is bounded and bounded away from $0$. But then, up to a subsequence, $v_n^+\rightarrow sw,$ $s>0$, which contradicts the fact that $v_n\rightharpoonup0$.\\
By $(F_3)$, $\Phi(sw)=\frac{1}{2}s^2+\circ(s^2)$ as $s\rightarrow0$. Hence $0<\sup_{\widehat{X}(w)}\Phi<\infty$. Since $\Phi$ is weakly upper semicontinuous on $\widehat{X}(w)$ and $\Phi\leq0$ on $\widehat{X}(w)\cap X^-$, the supremum  is attained at some point $u_0$ such that $u_0^+\neq0$. So $u_0$ is a nontrivial critical point of $\Phi|_{\widehat{X}(w)}$ and hence $u_0\in \mathcal{M}$.
\item[\textbf{(2)}] Now we show that if $u\in\mathcal{M}$, then $u$ is the unique global maximum of $\Phi|_{\widehat{X}(u)}$.\\
              Let $u\in\mathcal{M}$ and $u+w\in\widehat{X}(u)$ with $w\neq0$. By definition of $\widehat{X}(u)$ we have $u+w=(1+s)u+v$, $s\geq-1$ and $v\in X^-.$
              By using the fact that $s(\frac{s}{2}+1)u+(1+s)v\in X(u)$ we obtain
              \begin{multline*}
                \Phi(u+w)-\Phi(u)=-\frac{1}{2}\|v\|^2+\\
                 \int_\Omega\Big[\bigl(s(\frac{s}{2}+1)u+(1+s)v\bigr)\cdot\nabla F(x,u)+F(x,u)-F(x,u+w)\Big].
              \end{multline*}
              We define $g$ on $[-1,\infty[$ by
              \begin{equation*}
                g(s):=\big(s(\frac{s}{2}+1)u+(1+s)v\big)\cdot\nabla F(x,u)+F(x,u)-F(x,u+w).
              \end{equation*}
              Since $u\neq0$, then in view of $(F_5)$ we have $g(-1)<0$. On the other hand we deduce from $(F_4)$ and $(F_5)$ that $g(s)\rightarrow-\infty$ as $s\rightarrow\infty$. Assume that $g$ attains its maximum at a point $s\in[-1,\infty[$, then
              \begin{equation}\label{gprim}
                g'(s)=\big((1+s)u+v\big)\cdot\nabla F(x,u)-u\cdot\nabla F(x,(1+s)u+v)=0.
              \end{equation}
              Setting $z=u+w=(1+s)u+v$, one can easily verify that
              \begin{equation*}
                g(s)=-\big(\frac{s^2}{2}+s+1\big)u\cdot\nabla F(x,u)+(1+s)z\cdot\nabla F(x,u)+F(x,u)-F(x,z).
              \end{equation*}
              It is then clear that if $u\cdot z\leq0$, then $(F_6)$ implies $g(s)<0$. Suppose that $u\cdot z>0$, then in view of \eqref{gprim}, $(F_8)$ implies $|u|=|z|$ and  by $(F_7)$ we have $F(x,u)=F(x,z)$ and $z\cdot\nabla F(x,u)<u\cdot\nabla F(x,u)$ whenever $w\neq0$. This implies that $g(s)<-\frac{s^2}{2}u\cdot\nabla F(x,u)\leq0.$ Hence $\Phi(u+w)<\Phi(u)$. \medskip
        \end{proof}

\begin{lemma}\label{a3}
Assume $(F_2)-(F_8)$. Then $(A_3)$ is satisfied.
\end{lemma}
\begin{proof}
Clearly $(F_3)$ implies $I'(u)=\circ(\|u\|)$ as $|u|\rightarrow0$, which together with $(A_1)$ imply that
\begin{equation*}
    \forall \varepsilon>0, \exists\alpha>0\, |\, \forall u\in X^+, |u|<\alpha\Rightarrow I(u)<\frac{1}{2}\big<I'(u),u\big>\leq \|I'(u)\|\|u\|\leq\frac{\varepsilon}{2}\|u\|^2.
\end{equation*}
Hence we can find $\rho,\eta>0$ such that $\Phi(w)\geq\eta$ for any $w\in\{u\in X^+\,|\, \|u\|=\rho\}$. By $(A_2)$, $\Phi(\widehat{m})\geq \eta$ for any $w\in X\backslash X^-$. Since $I\geq0$, we deduce from \eqref{phi} that $\|\widehat{m}(w)^+\|\geq\sqrt{2\eta}$ for any $w\in X\backslash X^-$.
\par Now let $\mathcal{K}$ be a compact subset of $X\backslash X^-$. We want to show that there exists a constant $C_\mathcal{K}$ such that $\|\widehat{m}(w)\|\leq C_\mathcal{K}$, $\forall w\in\mathcal{K}$. Since $\widehat{m}(w)=\widehat{m}(w^+/\|w^+\|)$ $\forall w\in X\backslash X^-$, we may assume that $\mathcal{K}\subset S^+$. Suppose by contradiction that there exists a sequence $(w_n)\subset \mathcal{K}$ such that $\|\widehat{m}(w_n)\|\rightarrow\infty.$ Since $\widehat{m}(w_n)\in \widehat{X}(w_n)$, we have $\widehat{m}(w_n)=\lambda_nw_n+v_n$, with $\lambda_n\geq0$ and $v_n\in X^-$. Since $\Phi(\widehat{m}(w_n))>0$, $\|w_n\|=1$ and $I\geq0$, we deduce from \eqref{phi} that $\lambda_n\geq\|v_n\|$. Hence $\lambda_n\rightarrow\infty$, which implies $|\lambda_nw_n+v_n|\rightarrow\infty$ as $n\rightarrow\infty$. By \eqref{phi} we have
\begin{align*}
  0<\frac{\Phi(\widehat{m}(w_n))}{\lambda_n^2} &= \frac{1}{2}-\frac{1}{2}\frac{\|v_n\|^2}{\lambda_n^2}-\int_\Omega \frac{F(x,\lambda_nw_n+v_n)}{\lambda_n^2} \\
   &= \frac{1}{2}-\frac{1}{2}\frac{\|v_n\|^2}{\lambda_n^2}-\int_\Omega \frac{F(x,\lambda_nw_n+v_n)}{|\lambda_nw_n+v_n|^2}\frac{|\lambda_nw_n+v_n|^2}{\lambda_n^2} \\
   &\leq \frac{1}{2}-\int_\Omega \frac{F(x,\lambda_nw_n+v_n)}{|\lambda_nw_n+v_n|^2}|w_n|^2.\quad \quad \quad \quad \quad \quad \quad \quad \quad \quad(\star)
\end{align*}
Since $\mathcal{K}$ is compact we have, by taking a subsequence if necessary that $w_n\rightarrow w\in S^+$ and $w_n\rightarrow w$ a.e on $\Omega$. Clearly $w\neq0$. Then by using $(F_4)$ and Fatou's Lemma, we deduce from $(\star)$ that $0\leq-\infty$; a contradiction. \medskip
\end{proof}
We need the following result:
\begin{lemma}\label{lemmeboundG}
Let $1\leq q,r<\infty$ and $G\in\mathcal{C}(\overline{\Omega}\times\mathbb{R}\times\mathbb{R})$ such that
\begin{equation*}
    |G(x,a,b)|\leq c\bigl(1+|a|^{\frac{q}{r}}+|b|^{\frac{q}{r}}\bigr).
\end{equation*}
Then $\forall a,b\in L^q(\Omega)$, $G(\cdot,a,b)\in L^r(\Omega)$ and the operator $A:L^q(\Omega)\times L^q(\Omega)\rightarrow L^r(\Omega),(a,b)\mapsto G(x,a,b)$ is continuous.
\end{lemma}
The proof of Lemma \ref{lemmeboundG} follows the lines of the proof of Theorem $A.2$ in \cite{W} and is omitted here.
\begin{lemma}\label{pscondition}
Assume $(F_1)-(F_8)$. Then $\Phi$ satisfies the Palais-Smale condition on $\mathcal{M}$.
\end{lemma}
\begin{proof}
Let $(u_n)\subset \mathcal{M}$ be a sequence such that $\Phi(u_n)\leq d$ for some $d>0$ and $\Phi'(u_n)\rightarrow0$. We want to show that $(u_n)$ has a convergent subsequence.
\par Let us first show that $(u_n)$ is bounded.\\
If $(u_n)$ is not bounded, then up to a subsequence we have $\|u_n\|\rightarrow\infty.$ Define $v_n:=u_n/\|u_n\|$. We easily deduce from \eqref{phi} that
\begin{align*}
  0<\frac{\Phi(u_n)}{\|u_n\|^2} &= \frac{1}{2}\|v_n^+\|^2-\frac{1}{2}\|v_n^-\|^2-\frac{I(u_n)}{\|u_n\|^2} \\
   &=  \frac{1}{2}\|v_n^+\|^2-\frac{1}{2}\|v_n^-\|^2-\int_\Omega |v_n|^2 \frac{F(x,v_n\|u_n\|)}{\big|v_n\|u_n\|\big|^2}. \quad \quad \quad (\star \star)
\end{align*}
Since $(v_n)$ is bounded we have, by taking a subsequence if necessary, $v_n\rightharpoonup v$.  If $v\neq0$, then by using one more time $(F_4)$ and Fatou's Lemma we obtain from $(\star \star)$ the contradiction $0\leq-\infty$. Hence $v=0$. Since $\Phi(u_n)>0$ and $I(u_n)>0$, \eqref{phi} implies $\|v_n^+\|\geq \|v_n^-\|$. Hence we cannot have $v_n^+\rightarrow0$ (since $\|v_n\|=1$). There then exists $\alpha>0$ such that, up to a subsequence, $\|v_n^+\|\geq\alpha$ $\forall n$. It is clear that $sv_n^+\in \widehat{X}(u_n)$ $\forall s>0$. Then by $(A_2)$ we have $d\geq \Phi(u_n)\geq\Phi(sv_n^+)\geq \frac{1}{2}s^2\alpha^2-I(sv_n^+)$ $\forall s>0$. Since $v_n^+\rightharpoonup0,$ we deduce from the compactness of the embedding $X\hookrightarrow L^p(\Omega)\times L^p(\Omega)$ that $v_n^+\rightarrow0$ in $L^p(\Omega)\times L^p(\Omega)$. Now since by $(F_2)$ $F$ satisfies the conditions of Lemma \ref{lemmeboundG} (with $q=p$ and $r=1$), we deduce that  $I(sv_n^+)\rightarrow 0$. It then follows that $d\geq\frac{1}{2}s^2\alpha^2$ $\forall s>0$. This gives another contradiction if we take $s$ big enough. Hence $(u_n)$ is bounded.
\par By taking a subsequence if necessary we have $u_n\rightharpoonup u$ in $X$. It follows from the compactness of the embedding $X\hookrightarrow L^p(\Omega)\times L^p(\Omega)$ that $u_n\rightarrow u$ in $L^p(\Omega)\times L^p(\Omega)$.
 Now we easily obtain from \eqref{phi} and \eqref{phiprime}:
\begin{equation*}
    \|u_n^\pm-u^\pm\|^2=\pm\big<\Phi'(u_n)-\Phi'(u),u_n^\pm-u^\pm\big>\pm\int_\Omega (u_n^\pm-u^\pm)\cdot\big(\nabla F(x,u_n)-\nabla F(x,u)\big).
\end{equation*}
Clearly $\big<\Phi'(u_n)-\Phi'(u),u_n^\pm-u^\pm\big>\rightarrow 0$. By $(F_2)$ the components of $\nabla F$ satisfy the conditions of Lemma \ref{lemmeboundG} with $q=p-1$ and $r=\frac{p}{p-1}$, then by using the H\"{o}lder inequality and Lemma \ref{lemmeboundG} we obtain
$\int_\Omega (u_n^\pm-u^\pm)\cdot(\nabla F(x,u_n)-\nabla F(x,u)\big)\rightarrow 0$. Consequently $u_n\rightarrow u$. \medskip
\end{proof}

We also need the following consequence of the Ekeland variational principle:
\begin{lemma}[\cite{W}, Corollary $2.5$]\label{ekeland}
Let $E$ be a Banach space and let $\varphi\in\mathcal{C}^1(E,\mathbb{R})$ be bounded below. If $\varphi$ satisfies the Palais-Smale condition at level $\theta:=\inf_E \varphi,$ then there exists $x\in E$ such that $\varphi'(x)=0$ and $\theta=\varphi(x)$.
\end{lemma}

\begin{proof}[Proof of Theorem \ref{maintheorem}]
We already know from Lemmas \ref{a1}, \ref{a2} and \ref{a3} that $(A_1)$, $(A_2)$ and $(A_3)$ are satisfied. By Corollary \ref{reduction}-$(a)$  $\Psi\in C^1(S^+,\mathbb{R})$.
\par Let us show that $\Psi$ satisfies the Palais-Smale condition on $S^+$.\\
Let $(w_n)\subset S^+$ be a Palais-Smale sequence for $\Psi$. By Corollary \ref{reduction}-$(b)$ $(m(w_n))$ is a Palais-Smale sequence for $\Phi$ on $\mathcal{M}$. By Lemma \ref{pscondition} we have $m(w_n)\rightarrow w$ up to a subsequence. Since $m^{-1}$ is continuous, it follows that $w_n\rightarrow m^{-1}(w)$. Hence $\Psi$ satisfies the Palais-Smale condition on $S^+$. Particularly $\Psi$ satisfies the Palais-Smale condition at level $\theta=\inf_{S^+}\Psi$. By Corollary \ref{reduction}-$(c)$ $\inf_{S^+}\Psi=\inf_{\mathcal{M}}\Phi>0$ and $\Psi$ is bounded below. By Lemma \ref{ekeland} $\inf_{S^+}\Psi$ is a critical value of $\Psi$. There then exists $u_0\in S^+$ such that $\inf_{S^+}\Psi=\Psi(u_0)$ and $\Psi'(u_0)=0$. It follows from Corollary \ref{reduction}-$(c)$ that $m(u_0)$ is a critical point of $\Phi$ and $\Phi(m(u_0))=\inf_\mathcal{M}\Phi$. Hence $m(u_0)$ is a ground state solution for the equation $\Phi'(u)=0$. \medskip
\end{proof}

%

\end{document}